\documentclass[11pt]{article}

\usepackage{amsmath}
\usepackage{amssymb}
\usepackage{amsthm}
\usepackage{mathtools}

\usepackage[a4paper,margin=1in]{geometry}

\usepackage{authblk}
\usepackage{hyperref}
\usepackage{enumitem}

\usepackage{microtype}
\newtheorem{theorem}{Theorem}[section]
\newtheorem{lemma}[theorem]{Lemma}

\newtheorem{corollary}[theorem]{Corollary}
\newtheorem*{conjecture}{Conjecture}

\theoremstyle{definition}

\theoremstyle{remark}

\begin{document}

\title{Multicolor Ramsey and list Ramsey numbers for star-like trees}

\author[1]{Qinghong Zhao}
\author[2]{Yaping Mao\thanks{Corresponding author.}}
\author[3]{Xiangqian Zhou}

\affil[1]{\small School of Mathematical Sciences, Huaqiao University,
Quanzhou 362021, Fujian, China}

\affil[2]{\small Academy of Plateau Science and Sustainability and
School of Mathematics and Statistics, Qinghai Normal University,
Xining, Qinghai 810008, China}

\affil[3]{\small Department of Mathematics and Statistics,
Wright State University, Dayton, OH 45435, USA}

\date{}

\maketitle

\begin{abstract}
For a graph \(H\), the \(k\)-color Ramsey number \(r(H;k)\) is the least integer \(N\) such that every \(k\)-edge-coloring of \(K_N\) contains a monochromatic copy of \(H\). A \(k\)-list assignment has \(|L(e)|=k\) for every edge. The list Ramsey number \(r_\ell(H;k)\) is the least integer \(N\) for which there exists a \(k\)-list assignment on \(E(K_N)\) such that every coloring from the lists contains a monochromatic copy of \(H\). Let \(K_{1,n}\) be a star, \(S(n,m)\) the double star obtained by joining the centers of \(K_{1,n}\) and \(K_{1,m}\), and \(S_n^m\) the graph obtained from \(K_{1,n}\) by subdividing \(m\) edges once. Alon et al.\ conjectured that \(r_\ell(K_{1,n};k)=r(K_{1,n};k)\) for all \(k,n\ge1\). In this paper, we confirm their conjecture for all \(k\ge1\) and \(n\ge3\) by a unified direct proof. For even \(k\ge4\) and under explicit parameter conditions, we prove that \(r(S(n,m);k)=kn+m+2\) for even \(n\) and \(r(S_n^m;k)=k(n-1)+m+2\) for odd \(n\). For two colors, we establish a general list Ramsey lower bound and determine the common values of Ramsey and list Ramsey numbers for double and subdivided stars in explicit parameter ranges. These results close several gaps in the known bounds.


\end{abstract}

\footnotetext{Email-address: qzhao@hqu.edu.cn (Q. Zhao), yapingmao@outlook.com (Y. Mao), xiangqian.zhou@wright.edu (X. Zhou).}

\section{Introduction}
All graphs considered in this paper are finite, simple, and undirected,
unless otherwise stated. For a graph \(G\), let
\(V(G)\) and \(E(G)\) denote its vertex and edge sets. Let
\(N_G(v)\) and \(d_G(v)\) denote the neighborhood and degree of
\(v\in V(G)\), respectively, and let \(\Delta(G)\) denote the maximum
degree of \(G\). For disjoint sets \(X,Y\subseteq V(G)\), let
\(G[X]\) denote the subgraph induced by \(X\) and \(G[X,Y]\) the
bipartite subgraph consisting of the edges between \(X\) and \(Y\). We use \(K_n\) and \(K_{s,t}\)
for the complete graph and complete bipartite graph, respectively; in
particular, \(K_{1,n}\) is the star with \(n\) leaves. For an integer $k\ge1$, we define $[k]=\{1,\ldots,k\}$.

For a positive integer \(k\), a \(k\)-edge-coloring of a graph \(G\)
is a map \(\varphi:E(G)\to[k]\). The \(k\)-color Ramsey number
\(r(H;k)\) is the least integer \(N\) such that every
\(k\)-edge-coloring of \(K_N\) contains a monochromatic copy of \(H\). Two-color Ramsey problems with different target graphs, including stars and trees, have been extensively studied. Yan and
Peng~\cite{YanPengTreeStar} obtained tree-embedding conditions
and applied them to derive upper bounds and exact values for
tree--star Ramsey numbers. Zhang, Broersma, and
Chen~\cite{ZhangBroersmaChenCycleStar} determined cycle--star
Ramsey numbers for even cycle lengths in two explicit ranges.
Allen, {\L}uczak, Polcyn, and Zhang~\cite{AllenEtAlCycleStar}
later obtained an exact formula when the even cycle is
sufficiently long relative to the star. The diagonal multicolor Ramsey number of a star was determined
by Burr and Roberts \cite{BurrRoberts}.
\begin{theorem}[Burr and Roberts~\cite{BurrRoberts}]
\label{thm:star-ramsey}
For \(k\ge1\) and \(n\ge1\),
\[
 r(K_{1,n};k)=
 \begin{cases}
 k(n-1)+1,&\text{if \(k\) and \(n\) are even},\\
 k(n-1)+2,&\text{otherwise}.
 \end{cases}
\]
\end{theorem}

Two related families are double stars and subdivided stars. For
\(n\ge m\ge1\), the double star \(S(n,m)\) is obtained from disjoint
copies of \(K_{1,n}\) and \(K_{1,m}\) by joining their centers. For
\(n\ge2\) and \(1\le m\le n\), let \(S_n^m\) be obtained from
\(K_{1,n}\) by subdividing \(m\) edges once; thus
\(S_n^1=S(n-1,1)\). Grossman, Harary, and
Klawe~\cite{GrossmanEtAl} determined \(r(S(n,m);2)\) when
\(n\le\sqrt{2}m\) or \(n\ge3m\), as well as for every \(n\) when
\(m=1\). Flores Dub\'o and Stein~\cite{FloresStein} obtained a
further two-color upper bound with
\(\frac{1+\sqrt5}{2}m<n<3m\), while Bal, DeBiasio, and
Oren-Dahan~\cite{BalEtAl} obtained multicolor bounds for balanced
double stars.

Ruotolo and Song~\cite{RuotoloSong} studied both \(S(n,m)\) and
\(S_n^m\). Under explicit conditions on \(n,m,k\), they obtained
exact values for odd \(k\) and corresponding upper bounds for
every \(k\); for even \(k\), their bounds did not in general
coincide. In this paper, we give a lower-bound construction that applies to both
families and yields exact values in the following two even-color
ranges. For subdivided stars, we also improve the counting estimate
used for the upper bound.

\begin{theorem}
\label{thm:double-star}
Let \(k\ge4\) and \(n\ge m\ge1\), with \(k\) and \(n\) even. If $(n+1)\left\lceil\frac{n+1}{k-1}\right\rceil
 >m\bigl((k-1)n+m\bigr)$, then \(r(S(n,m);k)=kn+m+2\).
\end{theorem}

\begin{theorem}
\label{thm:subdivided-star}
Let \(k\ge4\) be even and let \(n\ge3\) be odd, with \(n\ge m\ge1\).
Let \(t=\lceil(n-m+1)/(k-1)\rceil\). If
\((n-m+1)t>m\bigl((k-1)(n-1)+1\bigr)\), then
\(r(S_n^m;k)=k(n-1)+m+2\). In particular, this equality holds
whenever \(n\ge m(k-1)^2+2m-2\).
\end{theorem}

An edge-list assignment \(L\)
assigns a set \(L(e)\) of colors to each edge \(e\), and an
\(L\)-edge-coloring is an edge-coloring \(\varphi\) satisfying
\(\varphi(e)\in L(e)\) for every edge \(e\). A \(k\)-list assignment
has \(|L(e)|=k\) for every edge. An edge-coloring is proper if adjacent edges receive distinct colors.
We denote by \(\chi'(G)\) the least number of colors in a proper
edge-coloring of \(G\), and by \(\chi'_{\ell}(G)\) the least integer
\(q\) such that \(G\) admits a proper \(L\)-edge-coloring for every
\(q\)-list assignment \(L\). The list Ramsey number
\(r_\ell(H;k)\), introduced by Alon, Buci\'c, Kalvari, Kuperwasser,
and Szab\'o~\cite{AlonEtAl}, is the least integer \(N\) for which
there exists a \(k\)-list assignment on \(E(K_N)\) such that every
coloring from the lists contains a monochromatic copy of \(H\).
Clearly, if all lists are equal, then \(r_\ell(H;k)\le r(H;k)\). Alon et al.~\cite{AlonEtAl}
initiated the systematic study of this parameter, and Fox, He, Luo,
and Xu~\cite{FoxEtAl} subsequently studied its growth for fixed
nonbipartite graphs. For stars, Alon et al.~\cite{AlonEtAl} posed the following
conjecture.

\begin{conjecture}[\cite{AlonEtAl}]
For all \(k\ge1\) and \(n\ge1\), \(r_\ell(K_{1,n};k)=r(K_{1,n};k)\).
\end{conjecture}

The cases \(k=1\) and \(n=1\) are obvious. Alon et al.~\cite{AlonEtAl} proved
the conjecture for \(k=2\), when \(k,n\) are both even, and for
each fixed \(k\) and all sufficiently large \(n\). For odd
\(n\ge3\), the improper list edge-coloring theorem of Hilton,
Stirling, and Slivnik~\cite{HiltonStirlingSlivnik}, applied with
the positive even integer \(s=n-1\), also gives the equality. Ruotolo and Song~\cite{RuotoloSong}, using a theorem of
Schauz~\cite{Schauz}, proved the conjecture whenever \(k\) is an
odd prime. We confirm their conjecture for all \(k\ge1\) and \(n\ge3\) by a unified direct proof based on a tournament construction.
\begin{theorem}
\label{thm:list-star}
For all \(k\ge1\) and \(n\ge3\),
\(r_{\ell}(K_{1,n};k)=r(K_{1,n};k)\).
\end{theorem}

For odd \(k\), the case \(n=2\) is equivalent to
\(\chi'_{\ell}(K_{k+1})=k\), which is the List Edge-Coloring Conjecture for $K_{k+1}$. Schauz~\cite{Schauz} proved this when
\(k\) is an odd prime.
The case of general odd \(k\) remains open.

Ruotolo and Song~\cite{RuotoloSong} also considered the two-color
list Ramsey numbers of double stars and subdivided stars. For double stars with
\(n\ge3m\), their estimates determined \(r_\ell(S(n,m);2)\)
except when \(n\) is odd and \(m\ge3\). We prove the following two-color transfer theorem for every graph
of maximum degree at least two, and obtain the exact value of
\(r_\ell(S(n,m);2)\) for all \(n\ge3m\).
\begin{theorem}
\label{thm:two-color-transfer}
\textup{(i)} Let \(H\) be a graph with \(d=\Delta(H)\geq2\). Then $r_{\ell}(H;2)\geq\min\{r(H;2),2d\}$.
In particular, if \(r(H;2)\leq2d\), then
\(r_{\ell}(H;2)=r(H;2)\).

\noindent
\textup{(ii)} Let \(n\geq m\geq1\) and \(n\geq3m\). Then
\[
 r_{\ell}(S(n,m);2)=r(S(n,m);2)
 =
 \begin{cases}
  2n+1,&\text{if \(n\) is odd and \(m\leq2\)},\\
  2n+2,&\text{otherwise}.
 \end{cases}
\]
\end{theorem}

For subdivided stars, Ruotolo and Song~\cite{RuotoloSong} settled the case
\(m=1\) and proved the two-color Ramsey and list Ramsey equalities
for \(m\in\{2,3\}\) under explicit bounds on \(n\). They further suggested that
\(r_\ell(S_n^m;2)=r(S_n^m;2)\) for every fixed \(m\ge4\)
and all sufficiently large \(n\).
We determine the common value
under explicit bounds on \(n\).
\begin{theorem}
\label{thm:two-color-subdivided}
Let \(m,n\) be integers with \(m\geq4\), and let
\(\varepsilon=0\) if \(n\) is odd and \(\varepsilon=1\) if \(n\)
is even. Suppose that $n\geq5m-7+2\varepsilon$ and $(2-\varepsilon)n\geq m^2-2$. Then $r_{\ell}(S_n^m;2)
 =r(S_n^m;2)
 =r(K_{1,n};2)
 =2n-\varepsilon$.
\end{theorem}

For further results and references
on this topic, we refer the readers to \cite{DasReedSkokan,HuEtAlBookCycle,Radziszowski}.

\section{Multicolor Ramsey numbers}
In this section, we first list two theorems that will be used in later proofs.
\begin{theorem}[K\H{o}nig's theorem~\cite{Konig}]
\label{thm:konig}
In every bipartite graph, the maximum size of a matching equals
the minimum size of a vertex cover.
\end{theorem}

\begin{theorem}[Ruotolo and
Song~\cite{RuotoloSong}]
\label{thm:ruotolo-song-upper}
Let \(k\ge2\) and \(n\ge m\ge1\) be integers. If
\((n+1)\left\lceil\frac{n+1}{k-1}\right\rceil
>m((k-1)n+m)\), then
\(r(S(n,m);k)\le kn+m+2\).
\end{theorem}

We next prove the following two useful lemmas.
\begin{lemma}\label{lem:cycle-cover}
Let \(k\ge4\) be even. There exist cycles \(C_c\), \(c\in[k]\), in
\(K_k\) such that \(V(C_c)=[k]\setminus\{c\}\) for every
\(c\in[k]\), and every edge of \(K_k\) belongs to exactly two of
these cycles.
\end{lemma}

\begin{proof}
Let \(r=(k-2)/2\) and \(q=k-1=2r+1\). Relabel the vertices of
\(K_k\) as \(\Omega=\mathbb Z_q\cup\{\infty\}\), with arithmetic
modulo \(q\). Define \(C_\infty=(0,1,\ldots,2r)\). For each
\(t\in\mathbb Z_q\), let \(v_{2s-1}(t)=t-s\) and
\(v_{2s}(t)=t+s\) for \(1\le s\le r\), and define
\(C_t=(\infty,v_1(t),\ldots,v_{2r}(t),\infty)\). Then
\(V(C_c)=\Omega\setminus\{c\}\) for every \(c\in\Omega\).
Each edge \(\infty x\) belongs to exactly \(C_{x+1}\) and
\(C_{x-r}\).

For finite vertices \(x,y\), define their cyclic distance to be the
unique \(d\in\{1,\ldots,r\}\) with \(y-x\equiv\pm d\pmod q\).
The finite edges of \(C_0\) are \(\{-s,s\}\), \(1\le s\le r\),
and \(\{s,-(s+1)\}\), \(1\le s\le r-1\). Since \(q\) is odd,
multiplication by \(2\) permutes the nonzero residues modulo \(q\)
up to sign. Thus the first family contains one edge of each cyclic
distance. The residues \(2s+1\), \(1\le s\le r-1\), are distinct
up to sign and avoid \(\pm1\), so the second family contains one
edge of each distance \(2,\ldots,r\).

For each \(d\in\{1,\ldots,r\}\), translating an edge of cyclic
distance \(d\) by every element of \(\mathbb Z_q\) gives all \(q\)
edges of that distance exactly once. Indeed, a nonzero translation
fixing an unordered pair would interchange its endpoints and imply
\(2d\equiv0\pmod q\), a contradiction. Hence each finite edge of
distance \(2,\ldots,r\) belongs to exactly two of the cycles
\(C_t\), \(t\in\mathbb Z_q\), and each edge of distance \(1\)
belongs to exactly one. The latter edges form \(C_\infty\).
Relabeling \(\Omega\) as \([k]\) completes the proof.
\end{proof}

For a colored graph, let \(d_c(x)\) denote
the number of edges of color \(c\) incident with \(x\).

\begin{lemma}\label{lem:even-coloring}
Let \(k\ge4\) and \(s\ge2\) be even integers, and let \(h\ge2\).
There exists a \(k\)-edge-coloring of \(K_{ks+h}\) such that, for
each color \(c\), all vertices of color-\(c\) degree at least
\(s+1\) belong to a color-\(c\) component of order \(s+h\).
\end{lemma}

\begin{proof}
Let \(G=K_{ks+h}\), and partition \(V(G)\) into disjoint sets
\(A,V_1,\ldots,V_k\), where \(|A|=h\) and \(|V_i|=s\) for each
\(i\in[k]\). Color all edges inside \(V_i\) and between \(A\)
and \(V_i\) with color \(i\), and color all edges inside \(A\)
with color \(1\).

Let \(C_1,\ldots,C_k\) be the cycles in Lemma~\ref{lem:cycle-cover}.
For \(1\le i<j\le k\), let \(a,b\) be the two distinct indices
with \(ij\in E(C_a)\cap E(C_b)\). Then \(a,b\notin\{i,j\}\).
Decompose \(G[V_i,V_j]\) into \(s\) perfect matchings. Color
\(s/2\) of them with color \(a\) and the other \(s/2\) with color
\(b\).

Fix \(c\in[k]\). If \(x\in V_i\) and \(i\ne c\), then \(i\)
has exactly two neighbors on \(C_c\). For each such neighbor \(j\),
exactly \(s/2\) edges from \(x\) to \(V_j\) have color \(c\).
No other edge incident with \(x\) has color \(c\), so
\(d_c(x)=s\). All edges between \(A\) and \(V_c\) have color
\(c\), and no edge of color \(c\) joins \(A\cup V_c\) to its
complement. Thus \(A\cup V_c\) is the vertex set of a
color-\(c\) component of order \(s+h\), and every vertex outside
this component has color-\(c\) degree \(s\), as required.
\end{proof}

\begin{corollary}\label{cor:component-isolation}
Let \(k\ge4\) and \(s\ge2\) be even integers, and let \(h\ge2\).
If \(H\) is connected, \(\Delta(H)\ge s+1\), and
\(|V(H)|>s+h\), then \(r(H;k)\ge ks+h+1\). In particular,
for every even \(k\ge4\), $r(S(n,m);k)\ge kn+m+2$ if \(n\ge m\ge1\) and \(n\) is even, and $r(S_n^m;k)\ge k(n-1)+m+2$ if \(n\ge3\) is odd and \(n\ge m\ge1\).
\end{corollary}

\begin{proof}
Consider the coloring in Lemma~\ref{lem:even-coloring}. Suppose
that it contains a copy of \(H\) of color \(c\). A vertex of
maximum degree in this copy has color-\(c\) degree at least
\(s+1\), and hence belongs to the color-\(c\) component of order
\(s+h\). Since \(H\) is connected, the entire copy belongs to this
component, contrary to \(|V(H)|>s+h\). Thus
\(r(H;k)\ge ks+h+1\). Taking \((s,h,H)=(n,m+1,S(n,m))\)
and \((s,h,H)=(n-1,m+1,S_n^m)\), respectively, proves the two
consequences.
\end{proof}

Theorem~\ref{thm:double-star} follows immediately from Corollary~\ref{cor:component-isolation} and Theorem~\ref{thm:ruotolo-song-upper}.
\medskip

Let \(t=\lceil(n-m+1)/(k-1)\rceil\). Ruotolo and
Song~\cite[Theorem~3.3]{RuotoloSong} proved
\(r(S_n^m;k)\le k(n-1)+m+2\) under the assumptions \(t>m\)
and \(nt>(t-m)(m-1)t+m((k-1)(n-1)+m)\).
The argument below replaces the term \((t-m)(m-1)t\) by
\((t-m)(m-1)\). The resulting condition is
\((n-m+1)t>m((k-1)(n-1)+1)\), which also implies \(t>m\).

\begin{lemma}\label{lem:subdivided-upper}
Let \(k\ge2\), \(n\ge2\), and \(n\ge m\ge1\), and let
\(t=\lceil(n-m+1)/(k-1)\rceil\). If
\((n-m+1)t>m((k-1)(n-1)+1)\), then
\(r(S_n^m;k)\le k(n-1)+m+2\).
\end{lemma}

\begin{proof}
The hypothesis implies \(t>m\); otherwise,
\((n-m+1)t\le m(n-m+1)\le m((k-1)(n-1)+1)\).
Let \(G=K_{k(n-1)+m+2}\) have a \(k\)-edge-coloring with no
monochromatic \(S_n^m\). Since every vertex has degree
\(k(n-1)+m+1>k(n-1)\), \(G\) contains a monochromatic
\(K_{1,n}\), say blue, with center \(x\) and leaf set \(A\).
Set \(B=V(G)\setminus(A\cup\{x\})\). Then
\(|B|=(k-1)(n-1)+m\). A blue matching of size \(m\) between
\(A\) and \(B\), together with the star, forms a blue \(S_n^m\).
By Theorem~\ref{thm:konig}, the blue edges between \(A\) and \(B\)
have a vertex cover \(U\) of size at most \(m-1\). Put
\(A'=A\setminus U\) and \(B'=B\setminus U\). Then
\(|A'|\ge n-m+1\), \(|B'|\ge(k-1)(n-1)+1\), and no edge
between \(A'\) and \(B'\) is blue.

Each vertex of \(A'\) has at least \(n\) neighbors of one color
in \(B'\). As there are \(k-1\) available colors, there exist
distinct vertices \(a_1,\ldots,a_t\in A'\) and sets
\(L_1,\ldots,L_t\subseteq B'\) such that \(|L_i|=n\) and all
edges between \(a_i\) and \(L_i\) are red for every \(i\in[t]\).
The sets \(L_i\) are indexed by their centers and need not be
distinct. For \(b\in B'\), let
\(p(b)=|\{i\in[t]:b\in L_i\}|\), and put
\(D=\{b\in B':p(b)\ge m+1\}\).

Suppose that \(|L_i\cap D|\ge m\) for some \(i\in[t]\).
Choose distinct \(b_1,\ldots,b_m\in L_i\cap D\). Each \(b_j\)
has at least \(m\) red neighbors in
\(\{a_1,\ldots,a_t\}\setminus\{a_i\}\). We may therefore
choose distinct \(c_1,\ldots,c_m\) from this set with \(b_jc_j\)
red for every \(j\in[m]\). The star with center \(a_i\) and
leaf set \(L_i\), together with these \(m\) edges, is a red
\(S_n^m\), a contradiction. Thus \(|L_i\cap D|\le m-1\)
for every \(i\in[t]\).

Set \(Q=\sum_{b\in D}p(b)\). Counting the pairs \((i,b)\) with
\(b\in L_i\cap D\) gives \(Q\le(m-1)t\), and
\(Q\le t|D|\). Since \(t>m\),
\begin{align*}
 nt=\sum_{b\in B'}p(b)
 &\le Q+m(|B'|-|D|)\\
 &\le \left(1-\frac mt\right)Q+m|B'|\\
 &\le (m-1)(t-m)+m\bigl((k-1)(n-1)+m\bigr)\\
 &= (m-1)t+m\bigl((k-1)(n-1)+1\bigr).
\end{align*}
Hence \((n-m+1)t\le m((k-1)(n-1)+1)\), contrary to the
hypothesis.
\end{proof}

\textbf{Proof of Theorem~\ref{thm:subdivided-star}.}
The lower bound follows from Corollary~\ref{cor:component-isolation},
and Lemma~\ref{lem:subdivided-upper} gives the matching upper bound.
Let \(a=k-1\),
\(n_0=ma^2+2m-2\), and
\(F(n)=(n-m+1)^2-ma(a(n-1)+1)\). Then
\(F(n_0)=ma(a-1)+(m-1)^2>0\), and, for \(n\ge n_0\),
\(F(n+1)-F(n)=2(n-m+1)+1-ma^2>0\). Thus \(F(n)>0\) for
\(n\ge n_0\). Since \(t\ge(n-m+1)/a\), it follows that
\((n-m+1)t\ge(n-m+1)^2/a>m(a(n-1)+1)\), as required.\qed

\section{Multicolor list Ramsey numbers}
In this section, we begin with three theorems that will be used in subsequent proofs.
\begin{theorem}[Landau's theorem~\cite{Landau}]\label{thm:landau}
The sequence \(p_1\leq p_2\leq\cdots\leq p_t\) of integers is the score sequence of a tournament of order \(t\) if and only if \(\sum_{i=1}^m p_i\ge\binom m2\) for $m\in[t]$, with equality for \(m=t\).
\end{theorem}

\begin{theorem}[Galvin's theorem~\cite{Galvin}]\label{thm:galvin}
Every bipartite multigraph \(G\) satisfies $\chi'_{\ell}(G)=\chi'(G)=\Delta(G)$.
\end{theorem}
\noindent
By Theorem~\ref{thm:galvin}, every bipartite multigraph \(G\) admits a proper \(L\)-edge-coloring for any edge-list assignment \(L\) satisfying $|L(e)|\geq\Delta(G)$ for every $e\in E(G)$.

\begin{theorem}[Grossman, Harary, and Klawe~\cite{GrossmanEtAl}]
\label{thm:ghk-double-star}
Let \(n\ge m\ge1\) be integers.

\noindent
\textup{(i)} If \(n\) is odd and \(m\le2\), then
\(r(S(n,m);2)=\max\{2n+1,n+2m+2\}\).

\noindent
\textup{(ii)} Suppose that \(n\le\sqrt{2}m\) or \(n\ge3m\).
If \(n\) is even or \(m\ge3\), then
\(r(S(n,m);2)=\max\{2n+2,n+2m+2\}\).
\end{theorem}

\textbf{Proof of Theorem~\ref{thm:list-star}.}
Let \(k\geq1\) and \(n\geq3\). If \(k=1\), then \(K_n\) does not
contain \(K_{1,n}\), whereas assigning the singleton list \(\{1\}\) to
every edge of \(K_{n+1}\) forces a monochromatic copy of \(K_{1,n}\).
Hence \(r_{\ell}(K_{1,n};1)=r(K_{1,n};1)=n+1\).

Assume that \(k\geq2\). Let \(\varepsilon=0\) if \(k\) and \(n\) are
both even, and let \(\varepsilon=1\) otherwise. Let
\(t=k(n-1)+\varepsilon\). By Theorem~\ref{thm:star-ramsey}, we have
\(r(K_{1,n};k)=t+1\). Since
\(r_{\ell}(K_{1,n};k)\leq r(K_{1,n};k)\), it remains to prove that
\(r_{\ell}(K_{1,n};k)>t\).

We first construct a tournament \(T\) on \(t\) vertices such that the
out-neighborhood and the in-neighborhood of every vertex can be
partitioned into sets of size at most \(k\), with \(n-1\) sets in
total.

Suppose first that \(n\) is even. Then \(n\geq4\), and \(t\) is even.
Let \(a=k(n-2)/2\) and \(b=kn/2+\varepsilon-1\). Since
\(b-a=k+\varepsilon-1>0\), let \(p_1,\ldots,p_t\) be the
nondecreasing sequence consisting of \(t/2\) terms equal to \(a\) and
\(t/2\) terms equal to \(b\). Since \(a+b=t-1\), we have
\(\sum_{i=1}^{t}p_i=t(a+b)/2=\binom{t}{2}\). For
\(1\leq m\leq t/2\), we have
\(\sum_{i=1}^{m}p_i-\binom{m}{2} =\frac{m\bigl(k(n-2)-m+1\bigr)}{2}\geq0\),
since
\(k(n-2)-m+1\geq k(n-2)-t/2+1
=(k(n-3)-\varepsilon+2)/2>0\). For \(t/2<m<t\), we have
\(\sum_{i=1}^{m}p_i-\binom{m}{2} =\frac{(t-m)(m-k-\varepsilon+1)}{2}\geq0\),
as
\(m-k-\varepsilon+1>t/2-k-\varepsilon+1
=(k(n-3)-\varepsilon+2)/2>0\). By
Theorem~\ref{thm:landau}, there exists a tournament \(T\) with this
score sequence. As \(a+b=t-1\), every vertex of outdegree \(a\) has
indegree \(b\), and every vertex of outdegree \(b\) has indegree \(a\).

If \(\varepsilon=1\), then \(a=k(n-2)/2\) and \(b=kn/2\). Thus a set
of size \(a\) can be partitioned into \((n-2)/2\) sets of size \(k\),
whereas a set of size \(b\) can be partitioned into \(n/2\) sets of
size \(k\). If \(\varepsilon=0\), then \(a=k(n-2)/2\) and
\(b=(n/2-1)k+(k-1)\). Hence a set of size \(a\) can be partitioned
into \((n-2)/2\) sets of size \(k\), whereas a set of size \(b\) can
be partitioned into \(n/2\) sets, of which \(n/2-1\) have size \(k\)
and one has size \(k-1\). Therefore, in either case, the
out-neighborhood and the in-neighborhood of each vertex can be
partitioned into sets of size at most \(k\), with
\((n-2)/2+n/2=n-1\) sets in total.

Suppose next that \(n\) is odd. Then \(\varepsilon=1\), and
\(t=k(n-1)+1\) is odd. Arrange the vertices \(v_1,\ldots,v_t\) in a
cyclic order, and orient the edges from each vertex toward the next
\((t-1)/2=k(n-1)/2\) vertices in this order. It follows that \(T\) is
regular, and every vertex has indegree and outdegree both equal to
\(k(n-1)/2\). Thus both its out-neighborhood and its in-neighborhood
can be partitioned into \((n-1)/2\) sets of size \(k\), and there are
\(n-1\) sets in total.

Let \(V(T)=\{v_1,\ldots,v_t\}\), and let \(L\) be an arbitrary
\(k\)-list assignment on \(E(K_t)\). Orient every edge of \(K_t\) as
in \(T\). For each \(j\in[t]\), partition the out-neighbors of \(v_j\)
into sets \(O_j^1,\ldots,O_j^{q_j^+}\), and partition the in-neighbors
of \(v_j\) into sets \(I_j^1,\ldots,I_j^{q_j^-}\), where every set
has size at most \(k\) and \(q_j^++q_j^-=n-1\).

For each \(j\in[t]\), let \(x_{j,+}^i\) be the vertex representing
\(O_j^i\), and let \(x_{j,-}^i\) be the vertex representing \(I_j^i\).
Construct an auxiliary graph \(G^\ast\) with vertex classes
\(X^+=\{x_{j,+}^i:j\in[t],\,1\leq i\leq q_j^+\}\) and
\(X^-=\{x_{j,-}^i:j\in[t],\,1\leq i\leq q_j^-\}\). For every arc
\(\overrightarrow{v_pv_q}\) of \(T\), let \(i_+\) and \(i_-\) be the
unique indices such that \(v_q\in O_p^{i_+}\) and
\(v_p\in I_q^{i_-}\), and add the edge
\(x_{p,+}^{i_+}x_{q,-}^{i_-}\) to \(G^\ast\). No other edges are
added. Thus every edge of \(G^\ast\) has one endpoint in \(X^+\) and
the other in \(X^-\), so \(G^\ast\) is bipartite.

The endpoints of each edge of \(G^\ast\) determine the corresponding
tail and head in \(T\). Hence the construction defines a bijection
between the arcs of \(T\) and the edges of \(G^\ast\). The edges
incident with \(x_{j,+}^i\) correspond precisely to the vertices in
\(O_j^i\), and the edges incident with \(x_{j,-}^i\) correspond
precisely to the vertices in \(I_j^i\). Since every block has size at
most \(k\), we have \(\Delta(G^\ast)\leq k\).

Transfer \(L\) to a list assignment \(L^\ast\) on \(G^\ast\) by
defining
\(L^\ast(x_{p,+}^{i_+}x_{q,-}^{i_-})=L(v_pv_q)\) whenever
\(\overrightarrow{v_pv_q}\) corresponds to
\(x_{p,+}^{i_+}x_{q,-}^{i_-}\). Every edge of \(G^\ast\) therefore
has a list of \(k\) colors. By Theorem~\ref{thm:galvin}, \(G^\ast\)
has a proper \(L^\ast\)-edge-coloring \(\varphi^\ast\). Transfer
\(\varphi^\ast\) back to \(K_t\) by defining
\(\varphi(v_pv_q)=
\varphi^\ast(x_{p,+}^{i_+}x_{q,-}^{i_-})\) for every arc
\(\overrightarrow{v_pv_q}\) of \(T\). By the definition of
\(L^\ast\), this is an \(L\)-edge-coloring of \(K_t\).

Fix \(v_j\in V(K_t)\) and a color \(\gamma\). Since
\(\varphi^\ast\) is proper, at most one edge of color \(\gamma\) is
incident with each auxiliary vertex associated with an outgoing or
incoming block of \(v_j\). Therefore, the number of edges of color
\(\gamma\) incident with \(v_j\) is at most
\(q_j^++q_j^-=n-1\). Thus no vertex is incident with \(n\) edges of
one color, and \(\varphi\) contains no monochromatic copy of
\(K_{1,n}\). Since \(L\) is arbitrary, we have
\(r_{\ell}(K_{1,n};k)>t\). Hence
\(r_{\ell}(K_{1,n};k)=r(K_{1,n};k)=t+1\), as required.\qed

\medskip
For an edge-coloring \(\varphi\), let \(d_{\varphi}(v,\alpha)\)
denote the number of edges of color \(\alpha\) incident with \(v\). A connected graph is called Eulerian if it has a closed trail
containing every edge exactly once.

\begin{lemma}
\label{lem:euler-tour-balancing}
Let \(G\) be a connected Eulerian graph with at least one edge, and
let \(L\) be a \(2\)-list assignment on \(E(G)\). Then \(G\) admits
an \(L\)-edge-coloring \(\varphi\) satisfying
\(d_{\varphi}(v,\alpha)\le d_G(v)/2\) for every vertex \(v\) and
every color \(\alpha\) if and only if \(|E(G)|\) is even or the
lists are not all equal.
\end{lemma}

\begin{proof}
Let \(M=|E(G)|\), and let \(e_1,\ldots,e_M\) be the cyclic order
of the edges on an Euler circuit. If the lists are equal and \(M\)
is even, color the edges alternately with the two colors in the
common list. If the lists are not all equal, two consecutive edges
have different lists. Relabeling cyclically, we may assume that
\(L(e_1)\ne L(e_M)\). Choose
\(\varphi(e_1)\in L(e_1)\setminus L(e_M)\), and then choose
\(\varphi(e_i)\in L(e_i)\setminus\{\varphi(e_{i-1})\}\) for
\(2\le i\le M\). Consecutive edges on the circuit receive different
colors, including \(e_M,e_1\). At each vertex \(v\), the circuit
pairs its incident edges into \(d_G(v)/2\) pairs of different
colors. Hence \(d_{\varphi}(v,\alpha)\le d_G(v)/2\).

Conversely, suppose that \(M\) is odd and every list is \(\{a,b\}\).
A coloring satisfying the stated inequalities would have
\(d_{\varphi}(v,a)=d_{\varphi}(v,b)=d_G(v)/2\) at every vertex.
Thus \(2|\varphi^{-1}(a)|=\sum_v d_{\varphi}(v,a)
=\tfrac12\sum_v d_G(v)=M\), a contradiction.
\end{proof}

\textbf{Proof of Theorem~\ref{thm:two-color-transfer}.}
For part~\textup{(i)}, let \(R=r(H;2)\),
\(s=\min\{R,2d\}-1\), and let \(L\) be any \(2\)-list assignment
on \(K_s\). Suppose first that \(s<2d-1\). Extend \(L\) to a
nonconstant \(2\)-list assignment on \(K_{2d-1}\); this is possible
because there is an edge outside \(K_s\). By
Lemma~\ref{lem:euler-tour-balancing}, there is a coloring in which
every vertex has degree at most \(d-1\) in each color. Its
restriction to \(K_s\) contains no monochromatic \(H\).

Suppose next that \(s=2d-1\). If \(|E(K_s)|\) is even or \(L\)
is nonconstant, the same lemma applies. Otherwise, since \(s<R\),
take an ordinary \(H\)-free two-coloring of \(K_s\) and relabel
its colors by the two colors in the common list. Thus every
\(2\)-list assignment on \(K_s\) has an \(H\)-free coloring, and
\(r_\ell(H;2)\ge s+1=\min\{R,2d\}\). If \(R\le2d\), equality
follows from \(r_\ell(H;2)\le R\).

For part~\textup{(ii)}, by Theorem~\ref{thm:ghk-double-star}, when \(n\ge3m\),
\(r(S(n,m);2)=2n+1\) if \(n\) is odd and \(m\le2\), and
\(r(S(n,m);2)=2n+2\) otherwise. In either case,
\(r(S(n,m);2)\le2\Delta(S(n,m))=2n+2\). By part~\textup{(i)}, the statement follows.\qed

For a red--blue complete graph, let \(G_R\) and
\(G_B\) denote its red and blue spanning subgraphs. For a vertex \(v\)
and a set \(X\), let
\(d_R(v,X):=|N_{G_R}(v)\cap X|\) and define \(d_B(v,X)\)
analogously. For disjoint vertex sets \(X\) and \(Y\), let
\(e_R(X,Y)\) denote the number of red edges between \(X\) and \(Y\).

\begin{lemma}\label{lem:extension}
Let \(n,q\) be integers with \(q\geq3\), let
\(\varepsilon,\eta\in\{0,1\}\), and suppose that
\(n\geq5q-7+2\varepsilon\).
Let \(G\) be a red--blue complete graph, and let \(A\) and \(B\) be
disjoint subsets of \(V(G)\) such that
\(|A|=n-q+1\), \(|B|=n-q-\varepsilon\),
and
\(\Delta(G_B[A])\leq q+\varepsilon-1\), and \(\Delta(G_B[B])\leq q-3\).
Let \(c\notin A\cup B\) and suppose that
\(d_R(c,A\cup B)\geq n-\eta\).
Suppose further that there are distinct vertices
\(u,v\notin A\cup B\cup\{c\}\) such that \(cuv\) is a red path. If
\(\eta=1\), suppose in addition that there is a vertex
\(w\notin A\cup B\cup\{c,u,v\}\) such that \(cw\) is red. Then
\(G\) contains a red copy of \(S_n^q\) centered at \(c\).
\end{lemma}

\begin{proof}
For \(X\in\{A,B\}\), let \(N_X:=N_{G_R}(c)\cap X\) and
\(M_X:=X\setminus N_X\), and define \(r_X:=|N_X|\),
\(t_X:=|M_X|\), \(D_A:=q+\varepsilon-1\), and \(D_B:=q-3\).
If \(\mu_X\) denotes the maximum size of a red matching between
\(N_X\) and \(M_X\), then
\(\mu_X\geq\min\{r_X,(t_X-D_X)_+\}\),
where \((z)_+:=\max\{z,0\}\). Indeed, if a maximum such matching
does not saturate \(N_X\), an uncovered vertex of \(N_X\) has at
least \((t_X-D_X)_+\) red neighbors in \(M_X\), all of which are
covered by the matching.

Let \(d:=r_A+r_B\), \(a:=\min\{q-1,d-n+1+\eta\}\), and
\(b:=q-1-a\).
Since \(d\geq n-\eta\), we have \(a\geq1\) and \(b\leq q-2\).
Moreover,
\(r_A\geq d-|B|\geq q+\varepsilon-\eta\geq b\) and \(r_B\geq d-|A|\geq q-1-\eta\geq b\).
Assume first that \(b>0\). Then \(a=d-n+1+\eta\), and hence
\(b=n+q-2-\eta-d\).
Therefore,
\begin{align*}
 (t_A-D_A)_++(t_B-D_B)_+
 &\geq t_A+t_B-D_A-D_B\\
 &=2n-4q+5-2\varepsilon-d\\
 &\geq n+q-2-\eta-d=b.
\end{align*}
Let \(p_X:=(t_X-D_X)_+\) for \(X\in\{A,B\}\).
Since \(r_A,r_B\geq b\), the preceding matching bound and
\(p_A+p_B\geq b\) imply
\(\mu_A+\mu_B\geq \min\{b,p_A\}+\min\{b,p_B\} \geq \min\{b,p_A+p_B\}=b\).
For each \(X\in\{A,B\}\), choose a red matching \(F_X\) of size
\(\mu_X\) between \(N_X\) and \(M_X\). Since \(A\cap B=\varnothing\),
\(F_A\cup F_B\) is a matching of size at least \(b\); let \(F\) be
any \(b\)-edge submatching of \(F_A\cup F_B\). If \(b=0\), let
\(F=\varnothing\).

Delete from \(N_A\cup N_B\) the \(b\) endpoints of \(F\) lying in
that set, and denote the remaining set by \(N'\). We claim that
\(G_R[N']\) contains a matching of size \(a\). Otherwise, deleting
the endpoints of a maximum matching leaves a red-independent set
whose intersections with \(A\) and \(B\) are blue cliques of orders
at most \(D_A+1\) and \(D_B+1\), respectively. Therefore,
\(d-b-2(a-1)\leq D_A+D_B+2\).
Since \(a+b=q-1\), it follows that
\(d-a\leq D_A+D_B+q-1=3q+\varepsilon-5\).
On the other hand,
\(d-a =\max\{d-q+1,n-1-\eta\} \geq n-1-\eta >3q+\varepsilon-5\),
where the strict inequality follows from
\(n\geq5q-7+2\varepsilon\), \(q\geq3\), and \(\eta\leq1\).
This is a contradiction.

The matching \(F\) and a red matching of size \(a\) in \(N'\)
form \(a+b=q-1\) red paths of length two beginning at \(c\) that are
pairwise vertex-disjoint outside \(c\). They use \(b+2a\) red neighbors of \(c\) in
\(A\cup B\), leaving at least
\(d-b-2a=d-q+1-a\geq n-q-\eta\) such neighbors unused. These paths,
the path \(cuv\), the edge
\(cw\) when \(\eta=1\), and \(n-q-\eta\) unused red neighbors of
\(c\) form a red copy of \(S_n^q\) centered at \(c\).
\end{proof}

\begin{lemma}\label{lem:one-step-subdivision}
Let \(n,q\) be integers with \(q\geq2\), and let
\(\varepsilon\in\{0,1\}\). Suppose that
\(n\geq3q+\varepsilon\), \(n\geq5q-7+2\varepsilon\), and \((2-\varepsilon)n\geq q^2-2\).
If a red--blue coloring of \(K_{2n-\varepsilon}\) contains a
monochromatic copy of \(S_n^{q-1}\), then it contains a
monochromatic copy of \(S_n^q\).
\end{lemma}

\begin{proof}
Let \(G\) denote the given red--blue colored copy of
\(K_{2n-\varepsilon}\), and suppose that \(G\) contains no
monochromatic copy of \(S_n^q\). Interchanging the colors if
necessary, let \(H\) be a red copy of \(S_n^{q-1}\), with
\(V(H)=\{x,y_1,\ldots,y_n,z_1,\ldots,z_{q-1}\}\) and
\(E(H)=\{xy_t:t\in[n]\}\cup\{y_tz_t:t\in[q-1]\}\).
Let \(A:=\{y_q,\ldots,y_n\}\), \(B:=V(G)\setminus V(H)\), and
\(C:=V(H)\setminus A\). Thus \(|A|=n-q+1\),
\(|B|=n-q-\varepsilon\), and \(|C|=2q-1\).
Every edge between \(A\) and \(B\) is blue, since a red edge
\(ab\), where \(a\in A\) and \(b\in B\), would replace the branch
\(xa\) of \(H\) by the path \(xab\), producing a red \(S_n^q\).

We first establish the following two bounds:
\begin{align}
 d_B\bigl(b,(B\setminus\{b\})\cup C\bigr)
 &\leq q-2
 &&(b\in B),\label{eq:local-B}\\
 d_B\bigl(a,(A\setminus\{a\})\cup C\bigr)
 &\leq q+\varepsilon-1
 &&(a\in A).\label{eq:local-A}
\end{align}
If \eqref{eq:local-B} fails, choose \(q-1\) of the indicated blue
neighbors of \(b\). Together with \(A\), these are \(n\)
blue neighbors of \(b\). Moreover,
\(|B|-q=n-2q-\varepsilon\geq q\), so we may choose \(q\) vertices
of \(A\) as subdivision vertices
and \(q\) distinct unused vertices of \(B\) as their terminal
vertices, producing a blue \(S_n^q\).

The proof of \eqref{eq:local-A} is symmetric. If it fails, use
\(B\) together with \(q+\varepsilon\) indicated blue neighbors of
\(a\), choose the subdivision vertices in \(B\), and choose their
terminal vertices in \(A\). This is possible because
\(|A|-1-(q+\varepsilon)=n-2q-\varepsilon\geq q\).

Every edge between \(x\) and \(B\) is blue. Indeed, suppose that
\(xb\) is red for some \(b\in B\). Every edge from \(b\) to
\(B\setminus\{b\}\) must then be blue, since, if \(bb'\) were red for some \(b'\in B\setminus\{b\}\),
deleting a short branch \(xa\), with \(a\in A\), and adding
the red path \(xbb'\) would produce a red \(S_n^q\). Hence
\(d_B\bigl(b,(B\setminus\{b\})\cup C\bigr)\geq |B|-1>q-2\),
contrary to \eqref{eq:local-B}.

If \(q=2\), then the blue edge \(xb\), for any \(b\in B\), implies
that \(d_B\bigl(b,(B\setminus\{b\})\cup C\bigr)\geq1\), whereas
\eqref{eq:local-B} implies that this degree is at most \(0\), a
contradiction. Hence we may assume that \(q\geq3\).

Let \(Y:=\{y_1,\ldots,y_n\}\) and
\(Z:=\{z_1,\ldots,z_{q-1}\}\). Let \(Q\) be the bipartite graph with
parts \(Y\) and \(Z\cup B\), whose edges are the red edges of \(G\)
between the two parts. The graph \(Q\) contains the matching
\(\{y_tz_t:t\in[q-1]\}\).
Let \(\nu(Q)\) denote the maximum size of a matching in \(Q\).
A matching of size \(q\) in \(Q\), together with the red edges from
\(x\) to \(Y\), would produce a red \(S_n^q\). Therefore,
\(\nu(Q)=q-1\).
By Theorem~\ref{thm:konig}, \(Q\) has a vertex cover \(U\)
of order \(q-1\). Since \(U\) covers the fixed matching, it contains
exactly one of \(y_t,z_t\) for every \(t\in[q-1]\), and no other
vertex.

For \(X\subseteq[q-1]\), define \(Y_X:=\{y_t:t\in X\}\) and
\(Z_X:=\{z_t:t\in X\}\). Let \(I:=\{t:y_t\in U\}\) and
\(J:=\{t:z_t\in U\}\), and let \(i:=|I|\) and \(j:=|J|\). Then
\(I\mathbin{\dot\cup}J=[q-1]\) and \(i+j=q-1\).
Since \(U\) is a vertex cover, every edge between \(A\) and \(Z_I\)
is blue, and every edge between \(B\) and \(Y_J\) is blue. Since
\(xb\) is blue for every \(b\in B\),
inequality~\eqref{eq:local-B} implies that \(j+1\leq q-2\), and hence
\(i\geq2\).

Let \(P:=Z_J\cup Y_I\).
By \eqref{eq:local-A}, every vertex of \(A\) has at least
\(i-\varepsilon\) red neighbors in \(P\), while
\eqref{eq:local-B} implies that every vertex of \(B\) has at least
\(j+2\) red neighbors in \(P\). Therefore,
\(e_R(P,A\cup B)\geq L:=(i-\varepsilon)(n-q+1) +(j+2)(n-q-\varepsilon)\).

Suppose that \(d_R(y_t,A\cup B)\leq n-2\) for \(t\in I\) and
\(d_R(z_t,A\cup B)\leq n-1\) for \(t\in J\).
Then
\(e_R(P,A\cup B)\leq U_0:=(n-2)i+(n-1)j\).
Using \(j=q-1-i\), we obtain
\[
\begin{aligned}
 L-U_0
 &=(2-\varepsilon)n-q^2-1-2\varepsilon
   +(2+\varepsilon)i\\
 &\geq(2-\varepsilon)n-q^2+3
 \geq1,
\end{aligned}
\]
contradicting the preceding lower and upper bounds. Therefore, there
is a vertex \(p\in P\) such that
\(d_R(p,A\cup B)\geq n-\eta(p)\),
where
\(\eta(p)=1\) for \(p\in Y_I\) and \(\eta(p)=0\) for \(p\in Z_J\).

The bounds \eqref{eq:local-B} and \eqref{eq:local-A}, together with
the blue edges between \(x\) and \(B\), imply
\(\Delta(G_B[A])\leq q+\varepsilon-1\), and \(\Delta(G_B[B])\leq q-3\).
If \(p=y_t\in Y_I\), choose \(s\in I\setminus\{t\}\). The red path
\(y_txy_s\) and the red edge \(y_tz_t\) satisfy the hypotheses of
Lemma~\ref{lem:extension} with \(c=p\), \(u=x\), \(v=y_s\),
\(w=z_t\), and \(\eta=1\). If \(p=z_t\in Z_J\), the red path
\(z_ty_tx\) satisfies the hypotheses of Lemma~\ref{lem:extension}
with \(c=p\), \(u=y_t\), \(v=x\), and \(\eta=0\). In either case,
Lemma~\ref{lem:extension} produces a red \(S_n^q\), a contradiction.
\end{proof}

\textbf{Proof of Theorem~\ref{thm:two-color-subdivided}.} By
Theorem~\ref{thm:list-star},
\(r_{\ell}(K_{1,n};2)=r(K_{1,n};2)=2n-\varepsilon\).
Since \(K_{1,n}\subseteq S_n^m\),
\(2n-\varepsilon =r_{\ell}(K_{1,n};2) \leq r_{\ell}(S_n^m;2) \leq r(S_n^m;2)\).
It remains to prove that \(r(S_n^m;2)\leq2n-\varepsilon\).

Let \(G\) be a red--blue coloring of \(K_{2n-\varepsilon}\). By Theorem~\ref{thm:ghk-double-star}, we have \(r(S_n^1;2)=2n-\varepsilon\). Since \(S_n^1=S(n-1,1)\), \(G\) contains a monochromatic copy of \(S_n^1\). For every \(2\le q\le m\), the hypotheses imply
\(n\ge5m-7+2\varepsilon\ge3m+\varepsilon\ge3q+\varepsilon\),
\(n\ge5q-7+2\varepsilon\), and
\((2-\varepsilon)n\ge q^2-2\).
Repeated application of Lemma~\ref{lem:one-step-subdivision}
therefore produces a monochromatic copy of \(S_n^m\). Hence
\(r(S_n^m;2)\leq2n-\varepsilon\).
Therefore,
\(r_{\ell}(S_n^m;2) =r(S_n^m;2) =r(K_{1,n};2) =2n-\varepsilon\).\qed

\section*{Acknowledgements}
\noindent
The authors declare that there is no conflict of competing interest.

\end{document}